\documentclass[a4paper]{amsart}

\usepackage{amssymb}
\usepackage{latexsym}
\usepackage{amsmath}
\usepackage{euscript}

      \def\dC{{\mathbb C}}

      \def\dR{{\mathbb R}}

\def\cD{{\mathcal D}}      
   \def\cH{{\mathcal H}}   
   \def\cK{{\mathcal K}}   \def\cL{{\mathcal L}}
\def\cM{{\mathcal M}}   \def\cN{{\mathcal N}}   \def\cO{{\mathcal O}}

\def\cal H{{\mathcal H}}

\def\R{\mathbb{R}}
\def\C{\mathbb{C}}

\def\N{\mathbb{N}}

\def\Imag{{\text{\rm Im}}}

\def\ran{{\text{\rm ran\,}}}
\def\dom{{\text{\rm dom\,}}}

\def\spann{\textup{span}}
\def\Res{\textup{Res}}

\def\phi{\varphi}
\def\d{{d}}
\def\eps{\varepsilon}

\DeclareMathOperator{\supp}{supp}

\newtheorem{theorem}{Theorem}[section]
\newtheorem*{thm*}{Theorem}
\newtheorem{proposition}[theorem]{Proposition}
\newtheorem{corollary}[theorem]{Corollary}
\newtheorem{lemma}[theorem]{Lemma}

\theoremstyle{definition}
\newtheorem{definition}[theorem]{Definition} 
\newtheorem{remark}[theorem]{Remark}

\newtheorem{assumption}[theorem]{Assumption}

\numberwithin{equation}{section}

\title[An inverse problem of Calder\'on type with partial data]{An inverse problem of Calder\'on type with partial data}

\author{Jussi Behrndt \and Jonathan Rohleder}

\address{Institut f\"ur Numerische Mathematik \\
Technische Universit\"at Graz \\
Steyrergasse 30\\
A-8010 Graz\\
Austria}
\email{behrndt@tugraz.at, rohleder@tugraz.at}

\begin{document}

\begin{abstract}
A generalized variant of the Calder\'on problem from electrical impedance tomography with partial data for anisotropic Lipschitz 
conductivities is considered in an arbitrary space dimension $n \geq 2$. The following two results are shown: (i) The selfadjoint 
Dirichlet operator associated with an elliptic differential expression on a bounded Lipschitz domain is determined uniquely up to 
unitary equivalence by the knowledge of the Dirichlet-to-Neumann map on
an open subset of the boundary, and (ii) the Dirichlet operator can be  reconstructed from the residuals of  the Dirichlet-to-Neumann 
map on this subset.
\end{abstract}

\maketitle

\section{Introduction and main results}

Let $\cL$ be a uniformly elliptic and formally symmetric second order differential expression of the form
\begin{align}\label{expression}
 \cL = - \sum_{j,k = 1}^n \partial_j a_{jk} \partial_k + \sum_{j = 1}^n \big( a_j \partial_j - \partial_j \overline a_j \big)  + a
\end{align}
with variable coefficients on a bounded Lipschitz domain $\Omega$ in $\R^n$, $n \geq 2$. The main objective of the present paper is to show that the selfadjoint {\em Dirichlet operator}
\begin{align}\label{dirichlet}
 A u = \cL u, \quad \dom A = \left\{ u \in H^1_0(\Omega) : \cL u \in L^2 (\Omega) \right\},
\end{align}
associated with $\cL$ in $L^2(\Omega)$ is uniquely determined by a local variant of the  Dirichlet-to-Neumann map 
on some open subset $\omega$ of the boundary $\partial\Omega$, and that $A$ can be reconstructed
from the residuals of this  partial Dirichlet-to-Neumann map. Here the
{\em Dirichlet-to-Neumann map} $M(\lambda)$ is defined by
\begin{align}\label{DNmap}
  M(\lambda):H^{1/2}(\partial \Omega)\rightarrow H^{-1/2}(\partial \Omega),\qquad  u_\lambda\vert_{\partial\Omega}\mapsto \partial_\cL u_\lambda |_{\partial \Omega},
\end{align}
where $u_\lambda \in H^1(\Omega)$ is a solution of $\cL u = \lambda u$ and $\partial_\cL u_\lambda |_{\partial \Omega}$ denotes the conormal derivative of $u$ at $\partial \Omega$. 
The mapping $M(\lambda)$ is well-defined for all $\lambda$ in the resolvent set $\rho(A)$ of $A$; see Section~2 for further details.

The above inverse problem is closely connected to and inspired by the famous Calder\'on problem from electrical impedance tomography, where the aim is to  determine the isotropic conductivity $\gamma$ of an inhomogeneous body from current and voltage data measured on the surface or on parts of it. 
The classical Calder\'on problem corresponds to the special case  $a_{jk}=\gamma\delta_{jk}$,
$a_j = a = 0$, $1 \leq j, k \leq n$, in \eqref{expression}, and the knowledge of
$M(\lambda)$  on $\partial \Omega$ or $\omega\subset\partial\Omega$ for some $\lambda$ or at $\lambda = 0$; see \cite{C80}.

The Calder\'on problem has been a major challenge in the field of inverse problems for PDEs in the last three decades. The first positive 
results were obtained for $\omega = \partial \Omega$ using only Dirichlet and Neumann data for $\lambda = 0$ in the pioneering work 
\cite{SU87} in dimension $n \geq 3$ and smooth $\gamma$, see also \cite{NSU88} for $\gamma\in C^{1,1}(\overline\Omega)$
and \cite{N88} for the reconstruction of $\gamma$ from the boundary measurements. In the two-dimensional case the first main contribution 
was the solution of the problem for $\gamma\in W^{2,p}(\Omega)$ in \cite{N96}; later in \cite{AP06} conductivities 
in $L^\infty(\Omega)$ were allowed. 
For partial data given only on special subsets of $\partial \Omega$ uniqueness was shown in the recent works \cite{BU02, KSU07} 
for a $C^2$-function $\gamma$ and a reconstruction method was provided in \cite{NS10}, see also \cite{IUY10} for a generalization 
in the two-dimensional case. 
Also the more general case of an anisotropic conductivity $(a_{jk})_{j,k = 1}^n$ has been investigated; in this situation, 
the single coefficients are in general not uniquely determined. Nevertheless, uniqueness up to diffeomorphisms was first shown for 
real-analytic coefficients assuming knowledge of $M(0)$ in \cite{LU89} in dimension $n \geq 3$ and in \cite{S90,SU91} for $n = 2$;
more general cases were treated in \cite{ALP05,DKSU09,SU03}. In the publications \cite{LTU03,LU01} the 
related problem of determining a real-analytic Riemannian manifold from the given Dirichlet and Neumann boundary data on 
arbitrary open subsets of $\partial \Omega$ was considered. There uniqueness up to isometry in $n \geq 3$ and uniqueness 
of the conformal class in $n = 2$ was shown for partial data supported in an open subset of $\partial \Omega$; see also
\cite{B88,B97,BK92,DKSU09,KK98,KKL01,KL06} for closely related problems as, e.g., the multidimensional Gelfand inverse spectral problem and
inverse problems for the wave equation with elliptic data.
For a detailed recent review and further references we also refer to \cite{U09}.

The aim of the present paper is to prove somewhat different, milder types of uniqueness and reconstruction 
results in space dimension $n \geq 2$ for partial data given on an open subset $\omega$ of $\partial \Omega$. 
Since the coefficients of $\cL$ are not uniquely determined in general we focus on the selfadjoint Dirichlet operator~\eqref{dirichlet} associated with $\cL$ on $\Omega$. In return this point of view onto the problem allows us to consider the more general differential expression~\eqref{expression} and to impose, in an arbitrary space dimension $n \geq 2$, the following comparatively weak assumptions on the coefficients of $\cL$.

\begin{assumption}\label{A1}
The coefficients $a_{jk}$ and $a_j$ are bounded Lipschitz functions on $\overline\Omega$ 
satisfying $a_{kj} = \overline{a_{jk}}$, $1 \leq j, k \leq n$, and $a \in L^\infty(\Omega)$ is real-valued. 
Moreover, $\cL$ is uniformly elliptic, i.e., $$\sum_{j,k=1}^n a_{jk}(x) \xi_j\xi_k\geq C \sum_{k=1}^n\xi_k^2 $$
holds for some $C>0$, all $\xi=(\xi_1,\dots\xi_n)^\top\in\dR^n$ and $x\in\overline\Omega$.
\end{assumption}

Furthermore, we impose the following conditions on the domain $\Omega$ and the subset $\omega$ of the boundary $\partial \Omega$ 
where Dirichlet and Neumann data is assumed to be given.

\begin{assumption}\label{A2}
$\Omega$ is a bounded Lipschitz domain in $\dR^n$, $n\geq 2$, and  $\omega$ is a nonempty, open subset of the boundary $\partial \Omega$.
\end{assumption}

The following two theorems on 
uniqueness and reconstruction of the Dirichlet operator $A$ in \eqref{dirichlet} from the knowledge of the
Dirichlet-to-Neumann map on $\omega\subseteq\partial\Omega$ are the main results in the present paper. 
The usual duality between
$H^{1/2}(\partial\Omega)$ and $H^{-1/2}(\partial\Omega)$ is denoted by $(\cdot,\cdot)$.

\begin{theorem}\label{uniqueness}
Let $\Omega$ and $\omega$ be as in Assumption~\ref{A2}, and let $\cL_1$ and $\cL_2$ be differential expressions on $\Omega$ of the form~\eqref{expression}
with coefficients $a_{jk,1}$, $a_{j,1}$, $a_1$ and $a_{jk,2}$, $a_{j,2}$, $a_2$ 
satisfying Assumption~\ref{A1}. Denote by $A_1$, $A_2$, and $M_1$, $M_2$ the corresponding Dirichlet operators and Dirichlet-to-Neumann maps, respectively. Assume that
\begin{align*}
\bigl(M_1(\lambda) \phi,\psi\bigr) = (M_2(\lambda) \phi,\psi)
\qquad \phi,\psi \in H^{1/2}(\partial \Omega), \,\,\supp \phi,\psi \subset \omega,
\end{align*}
holds for 
all $\lambda \in \cD$, where $\cD\subseteq \rho(A_1) \cap \rho(A_2)$ is a set of points which has an accumulation point in $\rho(A_1) \cap \rho(A_2)$. Then $A_1$ and $A_2$ are unitarily equivalent.
\end{theorem}

We point out that Theorem~\ref{uniqueness} is a ``mild'' uniqueness result in the sense that 
even in the special case $a_{jk}=\gamma\delta_{jk}$, $a_j=a=0$,
$1\leq j,k\leq n$ and $\omega=\partial\Omega$ in the classical Calder\'on problem it does not imply uniqueness
of $\gamma$ as shown in 
\cite{AP06,BU02,KSU07,N88,N96,NS10,NSU88,SU87}. Under additional smoothness assumptions Theorem~\ref{uniqueness} can also be 
viewed as a consequence of the gauge equivalence of second order elliptic operators on manifolds determined from
their boundary spectral data in the multidimensional Gelfand inverse spectral problem; cf. \cite{B88,BK92,KK98,KKL01,KL06}.

In order to state our second main result we remark that $\lambda\mapsto M(\lambda)$
is a meromorphic operator function; the residual of the Dirichlet-to-Neumann map $M$ at a pole $\mu$ is 
a linear operator from $H^{1/2}(\partial\Omega)$ to $H^{-1/2}(\partial \Omega)$ which is denoted by $\Res_\mu M$. Since the aim in the next theorem is to use only knowledge of $M$ and its residuals on $\omega$, also the operator $\Res_\mu M$ is replaced by a local version. 
The residual $\Res_\mu^\omega M$ on $\omega$  
is defined on the linear subspace $H^{1/2}_\omega(\partial\Omega)$ of functions $\phi \in H^{1/2} (\partial \Omega)$ with $\supp \phi \subset \omega$ by
\begin{align*}
 \Res_\mu^\omega M: H^{1/2}_\omega(\partial\Omega) \to \bigl(H^{1/2}_\omega(\partial\Omega)\bigr)', 
\quad (\Res_\mu^\omega M \phi) (\psi) := (\Res_\mu M \phi,\overline\psi\,),
\end{align*}
where $\phi,\psi\in  H^{1/2}_\omega(\partial\Omega)$.
The next theorem provides a one-to-one correspondence between the eigenfunctions of the Dirichlet operator $A$ and the residuals of the Dirichlet-to-Neumann map $M$ on $\omega$ and, hence, yields a representation of $A$ in terms of these residuals.

\begin{theorem}\label{reconstruction}
Let Assumption~\ref{A1} and Assumption~\ref{A2} be satisfied and let 
$M(\lambda)$ be the Dirichlet-to-Neumann map associated with the elliptic differential expression $\cL$. 
Then the poles of $\lambda\mapsto M(\lambda)$ coincide with the eigenvalues of the Dirichlet operator $A$ and for each eigenvalue $\lambda_k$ of $A$ 
the corresponding eigenspace is given by
\begin{align*}
 \ker (A - \lambda_k) = \ran \left( \tau_k^{-1} \circ \Res_{\lambda_k}^\omega M \right),
\end{align*}
where $\tau_k$ denotes the restriction of the Neumann trace operator $u\mapsto\partial_\cL u\vert_\omega$ 
onto $\ker (A - \lambda_k)$. In particular, there exist $\phi_1^{(k)}, \dots, \phi_{n(k)}^{(k)} \in 
H^{1/2}(\partial \Omega)$ with support in $\omega$, such that the functions 
$$e_i^{(k)} := \tau_k^{-1} \bigl(\Res_{\lambda_k}^\omega M\bigr) \phi_i^{(k)},\qquad i=1,\dots,n(k),$$ form an orthonormal basis of $\ker (A - \lambda_k)$ and $A$ can be represented in the form
\begin{align*}
 A u = \sum_{k = 1}^\infty \lambda_k \sum_{i = 1}^{n(k)} \big(u, e_i^{(k)}\big)_{\!L^2(\Omega)}\, e_i^{(k)}, \qquad u \in \dom A.
\end{align*}
\end{theorem}

The proofs  of our uniqueness and reconstruction results are based on the powerful interplay of modern operator theory with 
classical PDE techniques, as, e.g., unique continuation results from \cite{A57,AKS62,H83,W93}.
Some of the main ideas are inspired by abstract methods in extension and spectral theory of symmetric and selfadjoint 
operators as provided in \cite{AB09,BL07,BL10,DHMS06,DM91,DM95,DLS87,J00,LT77}. Further operator theoretic approaches 
to elliptic boundary value problems and related questions via Dirichlet-to-Neumann maps or other analytic operator 
functions can also be found in the recent publications \cite{AP04,AE10,AM08,BGW09,BHMNW09,BMNW08,GM08,GM09,GM11,GMZ07,GMZ07-2,R07,R09}. 
For general methods from extension theory of symmetric operators that are applied to elliptic PDEs we also refer to, e.g., 
 \cite{A00,EM03,GM09-2,G68,M10,MS07,P08,P07}, the monographs \cite{EE87,GBook09,LM72,M00}, 
and the references therein.

The main part of the present paper is devoted to the proofs of the two main results; along the proofs we show Proposition~\ref{eigenfunc} 
and Proposition~\ref{formulaEigenspace} on the eigenvalues and eigenspaces of the Dirichlet operator,
as well as a density result in Lemma~\ref{simplicity} which is of independent interest.
The paper closes with a short appendix, which summarizes some basic facts on unbounded operators in Hilbert spaces and on Banach space-valued 
analytic functions.

\section{Proofs of Theorem~\ref{uniqueness} and Theorem~\ref{reconstruction}}\label{proofs}

In this section we give complete proofs of the uniqueness and reconstruction theorems
from above. Instead of two single proofs the material is ordered in several smaller statements
which then lead to the proofs of the main results.

We fix some notation first. By $H^s(\Omega)$ and $H^s(\partial \Omega)$ we denote the Sobolev spaces of order $s \geq 0$ on $\Omega$ and 
$\partial \Omega$, respectively, and by $H_0^s(\Omega)$ the closure of the set of $C^\infty$-functions with compact support in $\Omega$ 
with respect to the $H^s$-norm. Further, $H^{-s} (\partial \Omega)$ denotes the dual space of $H^{s} (\partial \Omega)$; the duality is 
expressed via
\begin{equation}\label{duality}
f(\phi)=(f,\overline\phi)=\overline{(\overline\phi,f)},\qquad f\in H^{-s}(\partial \Omega),\,\,\,\phi\in H^{s}(\partial \Omega),
\end{equation}
which extends the $L^2$ inner product.
For the Lipschitz domain $\Omega$ we write $u |_{\partial \Omega}\in H^{1/2}(\partial \Omega)$ for the trace of $u \in H^1(\Omega)$ at 
the boundary $\partial \Omega$ and 
$\partial_\cL u |_{\partial \Omega}\in H^{-{1/2}}(\partial \Omega)$ for the conormal derivative 
or Neumann trace of $u$ 
at $\partial \Omega$ (with respect to the differential expression $\cL$), see, e.g., \cite[Chapter 4]{M00} for more details. 
Recall further that $H^1_0(\Omega)$ coincides with the kernel of the trace operator $u\mapsto u\vert_{\partial\Omega}$ on $H^1(\Omega)$.
In order to define a restriction 
of the conormal derivative to the nonempty, open subset $\omega\subset\partial\Omega$ 
let 
\begin{equation}\label{homega}
H_\omega^{1/2}(\partial\Omega)=\bigl\{\phi\in H^{1/2}(\partial\Omega):\supp\phi\subset\omega\bigr\}
\end{equation}
be the 
linear subspace of $H^{1/2} (\partial \Omega)$ which consists of functions with support in $\omega$.
The restriction $\partial_\cL u |_{\omega}$ of the Neumann trace 
$\partial_\cL u \vert_{\partial \Omega}$ to $\omega$  is defined as
\begin{equation}\label{partneumann}
 \bigl(\partial_\cL u |_{\omega}\bigr) (\phi) :=\bigl(\partial_\cL u |_{\partial \Omega}\bigr)(\phi)= 
\bigl(\partial_\cL u |_{\partial \Omega},\overline\phi \bigr), \qquad\phi \in  H_\omega^{1/2}(\partial\Omega).
\end{equation}

Let us recall some well-known properties of the Dirichlet operator associated to $\cL$, 
which can be found in, e.g., \cite[Chapter~VI]{EE87} and \cite[Chapter~4]{M00}.

\begin{proposition}\label{dirichletProp}
The Dirichlet operator $A$ in \eqref{dirichlet}
is a selfadjoint operator in $L^2(\Omega)$ and its spectrum $\sigma(A)$ consists of isolated (real) eigenvalues 
with finite-dimensional eigenspaces. The Dirichlet eigenvalues accumulate to $+\infty$ and are bounded from below.
\end{proposition}

The next lemma shows that the Dirichlet-to-Neumann map and the Poisson operator in Definition~\ref{poissondef} below are well-defined. 

\begin{lemma}\label{DNwelldef} 
For all $\lambda$ in the resolvent set $\rho(A)$ of $A$ and all $\phi \in H^{1/2}(\partial \Omega)$ there exists a unique solution $u_\lambda\in H^1(\Omega)$
of the boundary value problem
\begin{align}\label{BVP}
\cL u = \lambda u,\qquad  u |_{\partial \Omega} = \phi.
\end{align}
In particular, for all $\lambda\in\rho(A)$ the Dirichlet-to-Neumann map $M(\lambda)$ in \eqref{DNmap}
is well-defined.
\end{lemma}

\begin{proof}
Let $\lambda \in \rho(A)$ and $\phi \in H^{1/2} (\partial \Omega)$. Then the homogeneous problem
\begin{align*}
 (\cL - \lambda) u = 0, \quad u |_{\partial \Omega} = 0,
\end{align*}
has only the trivial solution, and by~\cite[Theorem~4.10]{M00} it follows that the inhomogeneous 
problem~\eqref{BVP} has a unique solution. 
\end{proof}

Besides the Dirichlet-to-Neumann map
a Poisson operator which maps functions on $\partial \Omega$
onto the corresponding solutions of \eqref{BVP} will play an important role. 

\begin{definition}\label{poissondef}
Let $\lambda \in \rho(A)$. The {\em Poisson operator} is defined by 
\begin{align}\label{Poisson}
\gamma(\lambda):H^{1/2} (\partial \Omega)\rightarrow L^2(\Omega),\qquad u_\lambda\vert_{\partial \Omega} \mapsto u_\lambda,
\end{align}
where $u_\lambda\in H^1(\Omega)$ is the unique solution of~\eqref{BVP}
with $\phi=u_\lambda|_{\partial\Omega}$. The range of the restriction of
the Poisson operator to $ H_\omega^{1/2}(\partial\Omega)$ is denoted by $\cN_\lambda$,
\begin{equation}\label{NLambda}
 \cN_\lambda = \left\{ u \in H^1(\Omega) : \cL u = \lambda u,\, \supp (u|_{\partial \Omega}) \subset \omega\right\}.
\end{equation}
\end{definition}

The operator $\gamma(\lambda)$ is well-defined for each $\lambda \in \rho(A)$ by Lemma~\ref{DNwelldef} and the relation $M(\lambda)\phi = \partial_\cL(\gamma(\lambda)\phi)\vert_{\partial \Omega}$ holds
for all $\phi\in H^{1/2} (\partial \Omega)$.
Some properties and formulas for the Poisson operator and the Dirichlet-to-Neumann map will be given in the next lemma. Its proof is essentially based on the second Green identity
\begin{align}\label{GI}
 (\cL u, v) - (u, \cL v) = \bigl( u |_{\partial \Omega}, \partial_\cL v |_{\partial \Omega} \bigr)
 - \bigl( \partial_\cL u |_{\partial \Omega}, v |_{\partial \Omega} \bigr)
\end{align}
for $u, v \in H^1(\Omega)$ satisfying $\cL u, \cL v \in L^2(\Omega)$, see, e.g.,~\cite{M00}. Here $(\cdot, \cdot)$ on the left hand side denotes the inner product in $L^2(\Omega)$ and on the right hand side the 
duality between $H^{-1/2} (\partial \Omega)$ and $H^{1/2} (\partial \Omega)$; cf. \eqref{duality}.
In the following it will be clear from the context whether the entries in $(\cdot,\cdot)$ are functions on $\Omega$
or $\partial\Omega$, respectively, so that no confusion can arise.  
 We remark that in 
a more abstract setting statements of similar form as in Lemma~\ref{lemma1} can be found in, e.g., \cite[Proposition~2.6]{BL07}, 
\cite[$\S$~1]{DM91}, and \cite[$\S$~2]{LT77}.

\begin{lemma}\label{lemma1}
Let $\lambda,\mu\in\rho(A)$, let $\gamma(\lambda)$, $\gamma(\mu)$ be the Poisson operators and let $M(\lambda)$, $M(\mu)$ be the Dirichlet-to-Neumann maps.  Then the following statements (i)-(iv) hold.
\begin{enumerate}
 \item $\gamma(\lambda)$ is bounded and the adjoint operator $\gamma(\lambda)^\prime:L^2(\Omega)\rightarrow H^{-1/2} (\partial \Omega)$ is given by
\begin{align*}
 \gamma(\lambda)' u = - \partial_\cL \left((A - \overline \lambda)^{-1} u \right)|_{\partial \Omega}, \qquad u \in L^2(\Omega).
\end{align*}
\item $\gamma(\lambda)$ and $\gamma(\mu)$ satisfy the identity
 \begin{align*}
  \gamma(\lambda) 
= \left( I + (\lambda - \mu) (A - \lambda)^{-1} \right) \gamma(\mu).
\end{align*}
 \item The Poisson operators and the Dirichlet-to-Neumann maps are connected via
 \begin{align*}
 (\overline\mu-\lambda)\bigl(\gamma(\lambda) \phi,\gamma(\mu)\psi\bigr) = (M(\lambda) \phi,\psi)-(\phi,M(\mu)\psi), 
 \end{align*}
and, in particular, $(M(\lambda)\phi,\psi)=(\phi,M(\overline\lambda)\psi)$ holds for all $\phi,\psi\in H^{1/2}(\partial \Omega)$.
\item $M(\lambda)$ is bounded, the function $\lambda\mapsto M(\lambda)$ is holomorphic on $\rho(A)$, and the identity 
\begin{equation*}
\begin{split}
(M(\lambda)\phi,\psi) & = (\phi,M(\lambda_0)\psi)\\ & \quad +
(\overline\lambda_0-\lambda)\Bigl(\bigl(I+(\lambda-\lambda_0)(A-\lambda)^{-1}\bigr)\gamma(\lambda_0)\phi,\gamma(\lambda_0)\psi\Bigr)
\end{split}
\end{equation*}
holds for all $\lambda,\lambda_0\in\rho(A)$ and $\phi,\psi\in H^{1/2}(\partial \Omega)$. 
In particular, every eigenvalue of $A$ is either a pole of order one or a removable singularity of 
the mapping $\lambda\mapsto M(\lambda)$.
\end{enumerate}
\end{lemma}

\begin{proof}
(i) In order to verify the formula for $\gamma(\lambda)^\prime$ we show first
$$ (\gamma(\lambda) \phi, u) = \bigl(\phi, - \partial_\cL \bigl((A - \overline \lambda)^{-1} u \bigr) |_{\partial \Omega}\bigr)
\quad\text{for all}\,\,\,\phi \in H^{1/2}(\partial \Omega),\,\,\,u \in L^2(\Omega).$$
Let $\phi \in H^{1/2} (\partial \Omega)$, set $u_\lambda = \gamma(\lambda) \phi$ and let 
$u \in L^2(\Omega)$. One finds by~\eqref{GI} 
\begin{align*}
 (\gamma(\lambda) \phi, u) &= \left(u_\lambda, \left(I + \overline\lambda (A - \overline\lambda)^{-1} \right) u \right) - \left(\lambda u_\lambda, (A - \overline\lambda)^{-1} u \right)\\
 & = \left( u_\lambda, \cL (A - \overline\lambda)^{-1} u \right) - \left( \cL u_\lambda, (A - \overline\lambda)^{-1} u \right)\\
 & = \left(\partial_\cL u_\lambda |_{\partial \Omega}, (A - \overline\lambda)^{-1} u |_{\partial \Omega} \right) - 
\left( u_\lambda |_{\partial \Omega}, \partial_\cL \left((A - \overline\lambda)^{-1} u\right) |_{\partial \Omega} \right) \\
&=\left( \phi, -\partial_\cL \left((A - \overline\lambda)^{-1} u\right) |_{\partial \Omega} \right),
\end{align*}
where we have used $(A - \overline\lambda)^{-1} u |_{\partial \Omega}=0$ in the last equality. 
Then it follows with \eqref{duality} from
\begin{equation*}
(\gamma(\lambda)^\prime u)(\varphi)=(u,\overline{\gamma(\lambda)\phi})=(u,\gamma(\lambda)\overline\phi)
=- \bigl(\partial_\cL \bigl((A - \overline \lambda)^{-1} u \bigr) |_{\partial \Omega}\bigr)(\phi)
\end{equation*}
that $\gamma(\lambda)^\prime$ acts as
in the assertion and is defined on $L^2(\Omega)$. Moreover, the above reasoning also implies that $\gamma(\lambda)$ is closed and
hence bounded by the closed graph theorem.  

(ii) For $\lambda, \mu \in \rho(A)$, $\phi\in H^{1/2} ({\partial \Omega})$ and $u \in L^2(\Omega)$ we find by~(i)
\begin{align*}
(\gamma(\lambda)\phi,u) - (\gamma(\mu)\phi,u) &= 
\bigl(\phi,- \partial_\cL ((A - \overline \mu)^{-1}(\overline \lambda - \overline \mu) (A - \overline \lambda)^{-1} u )\vert_{\partial \Omega}\bigr)\\ 
 &= \bigl(\gamma(\mu)\phi, (\overline \lambda - \overline \mu) (A - \overline \lambda)^{-1} u\bigr)\\
 &= \bigl((\lambda-\mu)(A-\lambda)^{-1} \gamma(\mu) \phi,u\bigr).
\end{align*}
Hence we have $\gamma(\lambda)\phi=\gamma(\mu)\phi+(\lambda-\mu)(A-\lambda)^{-1} \gamma(\mu) \phi$,
which shows (ii).

(iii) Let $\lambda, \mu \in \rho(A)$ and $\phi, \psi \in H^{1/2} ({\partial \Omega})$, and set $u_\lambda = \gamma(\lambda) \phi$ and $v_\mu = \gamma(\mu) \psi$. Again by~\eqref{GI} we find
\begin{align*}
(\overline\mu-\lambda)\bigl( \gamma(\lambda) \phi, \gamma(\mu) \psi \bigr) = 
\left(u_\lambda, \cL v_\mu\right)-\left(\cL u_\lambda,v_\mu \right) 
= (M(\lambda) \phi, \psi)-(\phi, M(\mu) \psi).
\end{align*}

(iv) From $(M(\lambda)\phi,\psi)=(\phi,M(\overline\lambda)\psi)$ for $\lambda\in\rho(A)$ and $\phi,\psi\in H^{1/2}({\partial \Omega})$ it follows that $M(\lambda):H^{1/2}({\partial \Omega})\rightarrow
H^{-1/2}({\partial \Omega})$ is closed and hence bounded by the closed graph theorem. Furthermore,
by (iii) we have
\begin{equation*}
(M(\lambda)\phi,\psi)=(\phi,M(\lambda_0)\psi)+
(\overline\lambda_0-\lambda)\bigl(\gamma(\lambda)\phi,\gamma(\lambda_0)\psi\bigr),
\end{equation*}
which together with (ii) implies the formula in (iv). The remaining statement in (iv)
follows from this and the corresponding properties of the resolvent of $A$, see also Appendix.
\end{proof}

The statement in Lemma~\ref{lemma1}~(iv) on the singularities of $M$ and the eigenvalues of 
$A$ will be improved later in Corollary~\ref{charEigenvalues}, where it turns out that the Dirichlet eigenvalues coincide with the poles of the 
meromorphic operator function $\lambda\mapsto M(\lambda)$ and its restriction on $\omega$.

The following proposition is essential for the proofs of our main results.
It states that the restriction of the Dirichlet operator $A$ onto 
$\{u\in \dom A:\partial_\cL u\vert_\omega=0\}$ is an operator without eigenvalues. In an operator theoretic language this
implies that this restriction is a {\it simple} symmetric operator. 
The main ingredient in the proof of Proposition~\ref{eigenfunc} is a classical unique continuation theorem for solutions of second order elliptic differential inequalities, see, e.g., \cite{W93} and \cite{A57,AKS62,H83}. Moreover, the proof makes use of the sesquilinear form $\Phi_\cL(\cdot, \cdot)$ induced by $\cL$ on $H^1(\Omega)$,
\begin{align*}
 \Phi_\cL (u, v) := \int_\Omega \Big( \sum_{j,k = 1}^n a_{jk} \partial_k u \cdot \overline{\partial_j v} + \sum_{j = 1}^n \big( a_j (\partial_j u)\cdot  \overline v + \overline a_j u\cdot \overline{\partial_j v}\big)  + a u \overline v \Big) \d x
\end{align*}
for $u, v \in H^1(\Omega)$, and the corresponding first Green identity
\begin{align}\label{Green1}
 \Phi_\cL ( u, v) = (\cL u, v) + (\partial_\cL u |_{\partial \Omega}, v |_{\partial \Omega})
\end{align}
for all $u, v \in H^1(\Omega)$ satisfying $\cL u\in L^2(\Omega)$; see~\cite[Chapter 4]{M00}.

\begin{proposition}\label{eigenfunc}
There exists no eigenfunction $u$ of the Dirichlet operator $A$ satisfying $\partial_\cL u\vert_\omega=0$.
\end{proposition}

\begin{proof}
Let $\widetilde\Omega\supsetneq \Omega$ be 
a bounded Lipschitz domain such that 
$\partial \Omega \setminus \omega\subset \partial \widetilde\Omega$ and 
$\widetilde \Omega \setminus \Omega$ contains an open ball $\cO$. We extend  the coefficients $a_{jk}$, $a_j$, and $a$ of 
the differential expression $\cL$ from~\eqref{expression} to functions $\widetilde a_{jk}$, $\widetilde a_j$, and $\widetilde a$ 
on $\widetilde \Omega$ such that Assumption~\ref{A1} holds for the
corresponding differential expression $\widetilde \cL$ in $\widetilde\Omega$ defined as in \eqref{expression}. 

Assume that there exists $\lambda$ and $u \neq 0$ in the domain of the Dirichlet operator~\eqref{dirichlet} with 
$\cL u = \lambda u$ on $\Omega$ and $\partial_\cL u |_\omega = 0$. Since we have $u |_{\partial\Omega} =  0$, we 
can extend $u$ by zero on 
$\widetilde \Omega \setminus \Omega$ to a function $\widetilde u \in H^1(\widetilde \Omega)$. Moreover, 
$\widetilde \cL \widetilde u \in L^2(\Omega)$ and $\widetilde \cL \widetilde u = \lambda \widetilde u$ holds. 
In fact, we compute for $\widetilde \phi \in C_0^\infty (\widetilde \Omega)$
\begin{align*}
 \bigl( \widetilde \cL \widetilde u - \lambda \widetilde u \bigr) (\widetilde \phi)  
= \Phi_{\widetilde \cL} (\widetilde u, \overline{\widetilde \phi}\,) - \lambda (\widetilde u, \overline{\widetilde \phi}\,),
\end{align*}
where the left hand side is understood in the sense of distributions
and the right hand side consists of integrals on $\widetilde \Omega$ with integrands vanishing outside of $\Omega$. Denoting the restriction of $\widetilde \phi$ to $\Omega$ by $\phi$, it follows with the help of the first Green identity~\eqref{Green1} that
\begin{equation*}
 \bigl( \widetilde \cL \widetilde u - \lambda \widetilde u \bigr) (\widetilde \phi)  
= \Phi_\cL (u, \overline\phi) - \lambda (u, \overline\phi) = (\cL u, \overline\phi) + 
\bigl(\partial_\cL u |_{\partial \Omega}, \overline \phi |_{\partial \Omega}\bigr) - \lambda (u, \overline\phi).
\end{equation*}
Since $\cL u=\lambda u$, $\supp (\overline \phi |_{\partial \Omega}) \subset \omega$
and $\partial_\cL u |_\omega = 0$ we conclude together with \eqref{partneumann}
\begin{equation*}
 \bigl( \widetilde \cL \widetilde u - \lambda \widetilde u \bigr) (\widetilde \phi)
=\bigl(\partial_\cL u |_{\partial \Omega}, \overline \phi |_{\partial \Omega}\bigr)=
\bigl(\partial_\cL u\vert_w\bigr)(\phi\vert_{\partial\Omega})=0.
\end{equation*}
Hence $\widetilde \cL \widetilde u \in L^2(\widetilde \Omega)$ and $\widetilde \cL \widetilde u = \lambda \widetilde u$ hold; in particular, 
$\widetilde u$ is locally in $H^2$, see, e.g.,~\cite[Theorem~4.16]{M00}. Furthermore, we obtain
\begin{align*}
 - \sum_{j,k = 1}^n \widetilde a_{jk} \partial_j \partial_k \widetilde u =  \Big(\lambda - \widetilde a + \sum_{j = 1}^n \partial_j \overline{ \widetilde a }_j \Big) \widetilde u + \sum_{k = 1}^n \Big( \sum_{j = 1}^n\partial_j \widetilde a_{jk} - \widetilde a_k + \overline{\widetilde a}_k \Big) (\partial_k \widetilde u).
\end{align*}
Since the functions $\widetilde a_{jk}$ and $\widetilde a_j$, $1 \leq j, k \leq n$, together with their derivatives of first order as well as $\widetilde a$ are bounded on $\Omega$ and $\widetilde u=0$ on $\widetilde\Omega\setminus\Omega$ there
exist constants $\alpha$ and $\beta$ such that
\begin{equation}\label{ineq}
 \bigg| \sum_{j,k = 1}^n \widetilde a_{jk} \partial_j \partial_k \widetilde u \bigg| \leq \alpha |\widetilde u| + \beta \sum_{k = 1}^n |\partial_k \widetilde u|
\end{equation}
holds a.e.\ on $\widetilde\Omega$. As $\widetilde u=0$ on $\cO$ it follows from the differential inequality \eqref{ineq} and classical unique continuation results that $\widetilde u$ vanishes identically on $\widetilde\Omega$; cf. \cite{W93}.
In particular, we conclude $u = 0$ on $\Omega$, a contradiction, since $u$ was chosen to be an eigenfunction of $A$.
\end{proof}

Our last preparatory lemma will establish, as a consequence of Proposition~\ref{eigenfunc}, a density statement on the 
ranges $\cN_\lambda$ of the Poisson operators $\gamma(\lambda)$ in~\eqref{Poisson} restricted to $H_\omega^{1/2}(\partial\Omega)$. Recall that $\cN_\lambda$ is the space of solutions $u$ of the 
boundary value problem~\eqref{BVP} which satisfy $\supp (u |_{\partial \Omega}) \subset \omega$.

\begin{lemma}\label{simplicity}
Let $\cO\subseteq\dC^+$ be an open set and let $\cO^*=\{\lambda\in\dC:\overline\lambda\in\cO\}$.
Then 
\begin{align*}
  \spann \bigl\{ \cN_\lambda : \lambda \in \cO\cup \cO^* \bigr\} = \spann \bigl\{\gamma(\lambda)\phi:
\phi\in  H_\omega^{1/2}(\partial\Omega),\, \lambda \in \cO \cup \cO^* \bigr\}
\end{align*}
is a dense subspace of $L^2(\Omega)$.
\end{lemma}

\begin{proof}
The proof consists of three separate steps. It makes use of two further operator realizations 
$S$ and $T$ of the differential expression $\cL$. We consider the restriction 
\begin{align*}
 S u = \cL u, \quad \dom S = \bigl\{ u \in H_0^1(\Omega) : \cL u \in L^2(\Omega), \partial_\cL u |_\omega
= 0 \bigr\},
\end{align*}
of the Dirichlet operator $A$ in $L^2(\Omega)$, which has no eigenvalues by Proposition~\ref{eigenfunc},
and we define the operator $T$ in $L^2(\Omega)$ by
\begin{align*}
 T u = \cL u, \quad \dom T = \bigl\{u \in H^1 (\Omega) : \cL u \in L^2(\Omega), \,
u\vert_{\partial\Omega}\in H_\omega^{1/2}(\partial\Omega)\bigr\}.
\end{align*}
It follows immediately that the Dirichlet operator $A$ is a restriction of $T$ and that the spaces $\cN_\lambda$ coincide with $\ker (T - \lambda)$. In the first step of this proof we show that these spaces are dense in the spaces $\ker(S^*-\lambda)$.
In the second step, which can also be found in a different form in~\cite{K49}, a 
selfadjoint restriction of $S$ in the orthogonal complement of  
$\spann\{\cN_\lambda : \lambda \in \cO \cup \cO^*\}$ is constructed. In the last step we 
then show that the spectrum of this selfadjoint operator is empty, which implies the assertion.

{\bf Step 1.} In this step we show that $\cN_\lambda=\ker(T-\lambda)$ are dense subspaces of 
$\ker(S^*-\lambda)$, $\lambda\in\cO\cup\cO^*$. For this we check first that $S=T^*$ holds. In fact, for $u \in \dom S$ and $v\in\dom T$ the second
Green identity~\eqref{GI} together with $u|_{\partial \Omega} = 0$, $\partial_\cL u |_\omega = 0$, and $\supp (v |_{\partial \Omega}) \subset \omega$ implies
\begin{displaymath}
 (T v, u) = (v, \cL u) + \bigl( v |_{\partial \Omega}, \partial_{\cL} u |_{\partial \Omega} \bigr) - 
\bigl( \partial_{\cL} v |_{\partial \Omega}, u |_{\partial \Omega} \bigr)= (v, \cL u);
\end{displaymath}
cf. \eqref{partneumann}.
Hence $u \in \dom T^*$ and $T^* u = \cL u = S u$ by the definition of the adjoint operator. 
For the converse inclusion let $u \in \dom T^*$. From $A \subseteq T$ we obtain $T^* \subseteq A^* = A$ and therefore $T^* u=\cL u$ and 
$u \in\dom A$. In particular, we have $u\in H^1(\Omega)$, $\cL u\in L^2(\Omega)$, and $u |_{\partial \Omega} = 0$. It remains to show $\partial_{\cL} u |_\omega = 0$. 
For $v\in\dom T$ we have $\supp (v\vert_{\partial\Omega})\subset\omega$ and from \eqref{partneumann} and~\eqref{GI} we obtain 
\begin{align*}
\bigl(\partial_\cL u\vert_\omega\bigr)(\overline{v |_{\partial \Omega}})= 
\bigl(\partial_\cL u |_{\partial\Omega},v |_{\partial \Omega}\bigr) = -(T^*u, v) - (u, T v) + 
\bigl( u |_{\partial \Omega},\partial_\cL v |_{\partial \Omega}\bigr) = 0.
\end{align*}
As $\overline{v|_{\partial\Omega}}$ runs through $ H_\omega^{1/2}(\partial\Omega)$ as 
$v$ runs through $\dom T$, it follows $\partial_{\cL} u |_\omega = 0$, hence $u \in \dom S$. 
We have shown $T^* = S$. This implies $\overline T=T^{**}=S^*$ and hence the spaces
$\cN_\lambda = \ker (T - \lambda)$ are dense in the spaces $\ker(S^* - \lambda)$.

{\bf Step 2.}
In this step we show that the Hilbert space 
\begin{equation}\label{cm}
\cM := \bigl(\spann \{ \cN_\lambda : \lambda \in \cO \cup \cO^* \} \bigr)^\perp 
\end{equation}
is invariant for $S$ and that $S_\cM := S \upharpoonright (\cM \cap \dom S)$ is a selfadjoint operator in $\cM$;
cf. \cite{K49}.
Observe first that by step 1 the symmetric operator $S=T^*$ is closed and hence $\ran(S-\lambda)$ is closed for all
$\lambda\in\dC\setminus\dR$; cf. Appendix. According to step 1 we also have  
$$\cM = \bigl(\spann \{ \ker(S^* - \lambda) : \lambda \in \cO \cup \cO^* \} \bigr)^\perp = \bigcap_{\lambda \in \cO \cup \cO^*} \ran(S - \lambda).$$  
 Let $u \in \cM \cap \dom S$. Then $u \in \ran (S - \lambda)$ for all $\lambda \in \cO \cup \cO^*$, i.e., for each $\lambda \in \cO \cup \cO^*$ there exists  $u_\lambda \in \dom S$ such that $(S - \lambda) u_\lambda = u$ holds.
This implies
\begin{align*}
 S u = S (S - \lambda) u_\lambda = (S - \lambda) S u_\lambda \in \ran (S - \lambda)
\end{align*}
for each $\lambda \in \cO \cup \cO^*$ and hence $S u \in \cM$. Therefore,
$S_\cM=S \upharpoonright (\cM \cap \dom S)$ is an operator in $\cM$ and since $S$ is symmetric
we conclude that the restriction $S_\cM$ is a symmetric operator in $\cM$. 
For the selfadjointness of $S_\cM$ in $\cM$ it is sufficient to show
\begin{equation}\label{ranm}
\ran(S_\cM-\nu)=\cM=\ran(S_\cM-\overline\nu)\quad\text{for some}\quad\nu\in\cO.
\end{equation}
For this let $u\in\cM$. We claim that $v:=(S-\nu)^{-1}u\in\dom S$ belongs to $\cM$. In fact, for 
each $\lambda\in\cO\cup\cO^*$, $\lambda\not=\nu$, the function
\begin{align*}
 v_\lambda := \frac{1}{\lambda-\nu} \left( (S - \lambda)^{-1} - (S-\nu)^{-1} \right) u 
\end{align*}
satisfies $(S-\lambda)v_\lambda =v$ and hence $v\in\ran(S-\lambda)$ for all $\lambda\in\cO\cup\cO^*$,
$\lambda\not=\nu$. In order to check that also $v\in\ran(S-\nu)$ holds, we choose a sequence
$(\lambda_n)_n\subseteq\cO$, $\lambda_n\not=\nu$, which converges to $\nu$. 
As above it follows that $(S-\lambda_n)^{-1}u\in\ran(S-\lambda)$ for all $\lambda\in\cO\cup\cO^*$, $\lambda\not=\lambda_n$, and, in particular, $(S-\lambda_n)^{-1}u\in\ran(S-\nu)$.
Since $S$
is a closed symmetric operator the estimates $\Vert (S-\nu)^{-1}\Vert\leq\vert\Imag\,\nu\vert^{-1}$ and $\Vert (S-\lambda_n)^{-1}\Vert\leq\vert\Imag\,\lambda_n\vert^{-1}$ hold, and hence
\begin{equation*}
v-(S-\lambda_n)^{-1}u=(S-\nu)^{-1}u-(S-\lambda_n)^{-1}u=(\nu-\lambda_n)(S-\nu)^{-1}(S-\lambda_n)^{-1}u
\end{equation*}
implies
\begin{equation}\label{glim}
 v=\lim_{n\rightarrow\infty}(S-\lambda_n)^{-1}u.
\end{equation}
Since $(S-\lambda_n)^{-1}u\in\ran(S-\nu)$ and $\ran(S-\nu)$ is closed
we conclude $v\in\ran(S-\nu)$ from \eqref{glim}. We have shown $v\in\cM$. Moreover,
$(S_\cM-\nu)v=(S-\nu)v=u$ and, hence, the first equality in \eqref{ranm} holds. The second equality
in \eqref{ranm} follows in the same way. Therefore $S_\cM=S\upharpoonright(\cM\cap\dom S)$ is a selfadjoint operator in $\cM$.

{\bf Step 3.} It follows from step 2 that the operator $S$ can be written as the direct orthogonal sum
$S_\cM\oplus S_{\cM^\bot}$ with respect to the decomposition $L^2(\Omega)=\cM\oplus\cM^\bot$, where $S_\cM$ is a selfadjoint operator in $\cM$ and 
$S_{\cM^\bot}=S\upharpoonright(\cM^\bot\cap\dom S)$ is a closed symmetric operator in $\cM^{\bot}$.
Moreover, the selfadjoint Dirichlet operator $A$ admits the decomposition $A=S_\cM\oplus A^\prime_{\cM^\bot}$, where $A^\prime_{\cM^\bot}$
is a selfadjoint extension of $S_{\cM^\bot}$.
Since the spectrum of $A$ consists only of eigenvalues 
(see Proposition~\ref{dirichletProp}) the same holds for the spectrum of the selfadjoint part $S_\cM$. 
Clearly, each eigenfunction of $S_\cM$ is also an eigenfunction of $S=S_\cM\oplus S_{\cM^\bot}$, but by Proposition~\ref{eigenfunc} this operator has no eigenfunctions. Therefore, the spectrum of the selfadjoint
operator $S_\cM$ in $\cM$ is empty which implies $\cM=\{0\}$; cf. Appendix. Hence
\begin{equation*}
 L^2(\Omega)=\cM^\bot=\overline{\spann} \bigl\{ \cN_\lambda : \lambda \in \cO \cup \cO^* \bigr\}
\end{equation*}
by \eqref{cm}, where $\overline{\spann}$ denotes the closed linear span. 
This completes the proof of Lemma~\ref{simplicity}.
\end{proof}

With these preparations we are ready to prove the uniqueness result Theorem~\ref{uniqueness}.

\begin{proof}[{\bf Proof of Theorem~\ref{uniqueness}}]
Let $\cL_1$, $\cL_2$ be elliptic differential expressions as in the theorem and let $A_1$, $A_2$ and $M_1$, $M_2$
be the corresponding selfadjoint Dirichlet operators and Dirichlet-to-Neumann maps, respectively. 
The associated Poisson operators from Definition~\ref{poissondef} will be denoted by $\gamma_1(\lambda)$
and $\gamma_2(\lambda)$, respectively, $\lambda\in\rho(A_1)\cap\rho(A_2)$.
Since $\cD\subseteq\rho(A_1)\cap\rho(A_2)$ in Theorem~\ref{uniqueness} is a set with an accumulation point in the intersection of the domains 
of holomorphy of the functions $M_1$ and $M_2$ (see Lemma~\ref{lemma1}~(iv)), 
it follows that $M_1$ and $M_2$ coincide on $\omega$, i.e., 
\begin{equation}\label{coincide}
 (M_1(\lambda) \phi, \psi)= (M_2(\lambda) \phi, \psi), \qquad \phi, \psi \in H_\omega^{1/2}(\partial\Omega), 
\end{equation}
holds for all $\lambda\in\rho(A_1)\cap\rho(A_2)$, and, in particular, for all $\lambda \in \C \setminus \R$. Here $ H_\omega^{1/2}(\partial\Omega)$ 
is the linear subspace of $H^{1/2}(\partial\Omega)$ which consists of functions with support in $\omega$;
cf. \eqref{homega}.

Let in the following $\lambda, \mu \in \C \setminus \R$, $\lambda \neq \overline \mu$. Lemma~\ref{lemma1}~(iii) 
and \eqref{coincide} yield
\begin{equation}\label{isometry}
\begin{split}
\bigl(\gamma_1(\lambda) \phi,\gamma_1(\mu)\psi\bigr) &= \frac{(M_1(\lambda) \phi,\psi) - (\phi,M_1(\mu)\psi)}{\overline \mu-\lambda} \\
&= \frac{(M_2(\lambda) \phi,\psi) - (\phi,M_2(\mu)\psi)}{\overline \mu-\lambda}
=\bigl(\gamma_2(\lambda) \phi,\gamma_2(\mu)\psi\bigr)
\end{split}
\end{equation}
for all  $\phi, \psi \in  H_\omega^{1/2}(\partial\Omega)$.
Next we define a linear mapping $V$ in $L^2(\Omega)$ on the domain  
$\dom V=\spann \{\gamma_1(\lambda)\phi : \phi\in H_\omega^{1/2}(\partial\Omega),\,\lambda \in \C \setminus \R \}$ by
\begin{align*}
 V : \sum_{j = 1}^\ell \gamma_1(\lambda_j) \phi_j \mapsto \sum_{j = 1}^\ell \gamma_2(\lambda_j) \phi_j,
\quad \lambda_j \in \C \setminus \R,\,\,\phi_j \in H_\omega^{1/2}(\partial\Omega),\,\,1 \leq j \leq \ell.
\end{align*}
Formula~(\ref{isometry}) yields that $V$ is a well-defined, isometric operator in $L^2(\Omega)$ with
$\ran V=\spann \{\gamma_2(\lambda)\phi :\phi\in  H_\omega^{1/2}(\partial\Omega),\, \lambda \in \C \setminus \R \}$. 
By Lemma~\ref{simplicity} the domain and range of $V$ are both dense subspaces of $L^2(\Omega)$, and hence the closure $U$ of $V$ in $L^2(\Omega)$ is a unitary operator in $L^2(\Omega)$. Obviously 
$U \gamma_1(\lambda) \phi = \gamma_2(\lambda) \phi$ holds for all $\lambda \in \C \setminus \R$ and $\phi \in  H_\omega^{1/2}(\partial\Omega)$. With the help of Lemma~\ref{lemma1}~(ii) one computes for $\lambda \neq \mu$ 
\begin{align*}
 U (A_1 - \lambda)^{-1} \gamma_1(\mu)\phi = \frac{U\gamma_1(\lambda) \phi- U\gamma_1(\mu)\phi}{\lambda - \mu} = (A_2 - \lambda)^{-1} U \gamma_1(\mu)\phi
\end{align*}
for $\phi \in  H_\omega^{1/2}(\partial\Omega)$. From Lemma~\ref{simplicity} we then conclude
\begin{align*}
 U (A_1 - \lambda)^{-1} = (A_2 - \lambda)^{-1} U
\end{align*}
and therefore $UA_1u=A_2Uu$ for all $u\in\dom A_1$, that is, $A_1$ and $A_2$ are unitarily equivalent.
\end{proof}

The next proposition is the key ingredient in the proof of the reconstruction result 
Theorem~\ref{reconstruction}. Here the residual $\Res_\lambda M:H^{1/2}(\partial\Omega)\rightarrow
H^{-1/2}(\partial\Omega)$ of the Dirichlet-to-Neumann
map at a pole $\lambda$ and a local version of it play the main role. Define the operator
\begin{equation}\label{resi1}
\Res_\lambda^\omega M: H^{1/2}_\omega(\partial\Omega)\rightarrow 
\bigl(H^{1/2}_\omega(\partial\Omega)\bigr)^\prime,\qquad\phi\mapsto \Res_\lambda^\omega M\phi,
\end{equation}
where $\Res_\lambda^\omega M\phi$ is the restriction of $\Res_\lambda M\phi\in H^{-1/2}(\partial\Omega)$ onto $H^{1/2}_\omega(\partial\Omega)$, i.e.,
\begin{equation}\label{resi2}
\bigl(\Res_\lambda^\omega M\phi\bigr)(\psi):=\bigl(\Res_\lambda M\phi\bigr)(\psi)=\bigl(\Res_\lambda M\phi,\overline\psi),
\quad\psi\in H^{1/2}_\omega(\partial\Omega);
\end{equation}
cf. the definition above Theorem~\ref{reconstruction} in the introduction.
Note that in the special case $\omega=\partial\Omega$
the operators $\Res_\mu M$ and $\Res_\mu^\omega M$ coincide.
 It turns out that the Neumann trace operator on $\omega$ maps
each eigenspace of the Dirichlet operator bijectively onto the range of the corresponding 
residual of the Dirichlet-to-Neumann map on $\omega$.

\begin{proposition}\label{formulaEigenspace}
For each $\lambda$ the mapping
\begin{align*}
 \tau_\lambda : \ker (A - \lambda) & \to \ran \Res_\lambda^\omega M,\quad u  \mapsto \partial_\cL u |_\omega,
\end{align*}
is an isomorphism and, in particular, 
\begin{align*}
 \ran \Res_\lambda^\omega M = \bigl\{ \psi  : \text{there~exists}\, u \in \ker(A - \lambda)\, \text{such that}\, \,\partial_\cL u|_\omega = \psi \bigr\}.
\end{align*}
\end{proposition}

As an immediate consequence of this proposition we obtain the following statement 
which complements Lemma~\ref{lemma1}~(iv).

\begin{corollary}\label{charEigenvalues}
A point $\lambda_0$ is an eigenvalue of $A$ if and only if $\lambda_0$ is a pole of the Dirichlet-to-Neumann map
on $\omega$. In this case the dimension of the eigenspace $\ker (A - \lambda_0)$ coincides with the dimension of the range of the operator $\Res_{\lambda_0}^\omega M$.
\end{corollary}

\begin{proof}[Proof of Proposition~\ref{formulaEigenspace}]
Note first that if $\lambda$ is not an eigenvalue of $A$, then the function $M$ 
is holomorphic at
$\lambda$ by Lemma~\ref{lemma1}~(iv) and hence the residual $\Res_\lambda M$ is zero.
Therefore the assertion holds in this case.

In the following let $\lambda$ ($\in\dR$) be an eigenvalue of $A$ and let 
$\tau_\lambda u=\partial_\cL u\vert_\omega$ be the Neumann trace operator on $\omega$ on $\ker(A-\lambda)$ defined in \eqref{partneumann}.
For $u \in \ker (A - \lambda)$ and some fixed $\mu\in\dC\setminus\dR$ we obtain for all $\psi\in H^{1/2}_\omega(\partial\Omega)$ by~Lemma~\ref{lemma1}~(i)
\begin{equation}\label{taurep}
\begin{split}
(\tau_\lambda u)(\psi) & = (\partial_\cL u\vert_{\partial\Omega},\overline\psi\,)=
\bigl(\partial_\cL \left( (A - \overline \mu)^{-1} A u - (A - \overline \mu)^{-1} \overline \mu u \right) |_{\partial\Omega},\overline\psi\,\bigr)\\ 
 & = \bigl((\lambda-\overline \mu) \partial_\cL \left( ( A - \overline\mu )^{-1} u \right)|_{\partial\Omega},\overline\psi\,\bigr) 
= \bigl((\overline \mu - \lambda) \gamma(\mu)' u,\overline\psi\,\bigr).
\end{split}
\end{equation}

Let us show that $\tau_\lambda$ is injective. Assume there exists $u \in \ker(A - \lambda)$, $u \neq 0$, such that
$\tau_\lambda u=0$, that is,
\begin{align} \label{Neumann_trace_0}
\bigl(\partial_\cL \left((A - \overline \mu)^{-1} u\right)|_\omega\bigr)(\psi)=0
\end{align}
for all $\psi\in H^{1/2}_\omega(\partial\Omega)$ by \eqref{taurep}.
Since $u \neq 0$, also $(A - \overline \mu)^{-1} u \neq 0$ holds, but 
$$(A - \lambda) (A - \overline \mu)^{-1} u = (A - \overline \mu)^{-1} (A - \lambda) u = 0$$ 
implies that $(A - \overline \mu)^{-1} u$ is an eigenfunction of $A$ with vanishing Neumann trace on $\omega$ by~(\ref{Neumann_trace_0}). This contradicts Proposition~\ref{eigenfunc}. Hence $\tau_\lambda$ is injective.

Next we show that $\tau_\lambda$ maps onto $\ran \Res_\lambda^\omega M$, which by \eqref{taurep}  is equivalent to 
\begin{align}\label{eigenspace_residual_equality}
 \ran \Res_{\lambda}^\omega M
=
 \ran \bigl(\gamma(\mu)' \upharpoonright \ker(A - \lambda) \bigr).
\end{align}
The inclusion $\subseteq$ in \eqref{eigenspace_residual_equality} will be shown first.
For this denote by $E_{\lambda}$ the spectral projection onto the eigenspace $\ker (A - \lambda)$ of $A$. Let 
$\eta \in \C \setminus \R$, $\eta \neq \overline \mu$, and $\cO$ be an open ball centered in $\lambda$ such that $\cO$ does not contain 
any further eigenvalues of $A$ and $\mu, \eta \notin \overline{\cO}$. Denote by $\Gamma_{\lambda}$ the boundary of $\cO$ and write the 
spectral projection $E_\lambda$
as a Cauchy integral; cf. Appendix. For $\phi,\psi \in H^{1/2}(\partial\Omega)$ we obtain with the help of Lemma~\ref{lemma1} (ii) and (iii)
\begin{align*}
 \bigl(E_{\lambda}\gamma(\eta) \phi,&\gamma(\mu)\psi\bigr)  = - \frac{1}{2 \pi i} \int_{\Gamma_{\lambda}} 
\bigl( (A - \zeta)^{-1} \gamma(\eta)\phi,\gamma(\mu)\psi)\bigr) \d \zeta \\
 &= - \frac{1}{2 \pi i} \int_{\Gamma_{\lambda}} \left( \frac{1}{\zeta-\eta}
(\gamma(\zeta) \phi,\gamma(\mu)\psi) -  \frac{1}{\zeta-\eta}(\gamma(\eta) \phi,\gamma(\mu)\psi) \right) \d \zeta\\
 & = \frac{1}{2 \pi i} \int_{\Gamma_{\lambda}} \left( \frac{(M(\zeta) \phi,\psi)}{(\eta-\zeta) (\overline \mu-\zeta)} + \frac{(\phi,M(\mu) \psi)}{(\overline \mu-\zeta) (\overline \mu-\eta)} + \frac{(M(\eta) \phi,\psi)}{(\zeta - \eta) (\overline \mu-\eta)} \right) \d \zeta;
\end{align*}
cf. \cite[$\S {\rm I}.1$]{DLS93}. Note that the second and third fraction in the integral are holomorphic in $\cO$ as functions of $\zeta$ and hence can be neglected. In the remaining fraction, we develop the function $\zeta \mapsto M(\zeta) $ into a Laurent series at $\lambda$. Since this function has a pole of, at most, order one in $\lambda$, see Lemma~\ref{lemma1}~(iv) and the Appendix, we obtain
\begin{equation*}
 \bigl( E_{\lambda} \gamma(\eta) \phi,\gamma(\mu)\psi\bigr)  = \frac{1}{2 \pi i} \int_{\Gamma_{\lambda}} \frac{(\Res_{\lambda} M \phi,\psi)}{(\zeta - \lambda)(\eta - \zeta) (\overline \mu-\zeta) } \d \zeta = \frac{(\Res_{\lambda} M \phi,\psi)}{(\eta - \lambda)(\overline \mu-\lambda)},
\end{equation*}
where Cauchy's integral formula was used in the second equality.
Therefore we have
\begin{equation*}
(\Res_{\lambda} M \phi,\psi) = (\eta - \lambda)(\overline\mu-\lambda)\bigl( E_{\lambda} \gamma(\eta) \phi,\gamma(\mu)\psi\bigr) 
\end{equation*}
and hence, in particular,
\begin{equation}\label{abc}
\Res_\lambda^\omega M\phi=(\eta - \lambda)(\overline\mu-\lambda)\gamma(\mu)^\prime E_\lambda\gamma(\eta)\phi,
\qquad \phi\in H_\omega^{1/2}(\partial\Omega);
\end{equation}
cf. \eqref{resi1} and \eqref{resi2}.
This shows the inclusion $\subseteq$ in \eqref{eigenspace_residual_equality}.

In order to show the inclusion $\supseteq$ in \eqref{eigenspace_residual_equality} let 
$u\in\ker(A-\lambda)$ and $\eps >0$. Since $\gamma(\mu)' E_{\lambda}$ is bounded by Lemma~\ref{lemma1}~(i), there exists $\delta > 0$ such that $\| u - v \| < \delta$ implies $\| \gamma(\mu)' E_{\lambda} u - \gamma(\mu)' E_{\lambda} v \| < \eps$. 
Note that $E_\lambda u=u$ as $u\in\ker(A-\lambda)$. By Lemma~\ref{simplicity} we find $\ell \in \N$, $\eta_j \in \C \setminus \R$, and $\phi_j \in  H_\omega^{1/2}(\partial\Omega)$, $1 \leq j \leq \ell$, such that
\begin{align*}
 \bigg\| u - \sum_{j = 1}^\ell \gamma(\eta_j) \phi_j \bigg\| < \delta.
\end{align*}
Then we have
\begin{equation}\label{hurra}
\bigg\|\gamma(\mu)^\prime E_\lambda u - \gamma(\mu)^\prime E_\lambda  \sum_{j = 1}^\ell \gamma(\eta_j) \phi_j\bigg\|<\eps
\end{equation}
and from \eqref{abc} and \eqref{hurra} we conclude 
\begin{equation*}
\bigg\|\gamma(\mu)^\prime u - \sum_{j=1}^\ell \frac{\Res_\lambda^\omega M\phi_j}{(\eta_j-\lambda)(\overline\mu-\lambda)}\bigg\|<\eps,
\end{equation*}
that is, 
$\gamma(\mu)' u \in \overline{\ran \Res_{\lambda}^\omega M}$. 
The fact that $\ker(A - \lambda)$ is finite-dimensional (see~Proposition~\ref{dirichletProp}) 
shows~\eqref{eigenspace_residual_equality}. Thus $\tau_\lambda$ is bijective and maps the finite-dimensional
space $\ker(A - \lambda)$ onto $\ran \Res_{\lambda}^\omega M$, i.e., $\tau_\lambda$ is an isomorphism.
\end{proof}

Theorem~\ref{reconstruction} is now essentially a consequence of Proposition~\ref{formulaEigenspace}.

\begin{proof}[{\bf Proof of Theorem~\ref{reconstruction}}]
For the statement on the poles of $M$ see Corollary~\ref{charEigenvalues}. In order to prove the representation of $A$ let us denote by $(\lambda_k)_k$ the sequence of eigenvalues of $A$ and by $\tau_k$ the restriction of the Neumann trace operator onto the eigenspace $\ker(A - \lambda_k)$ from Proposition~\ref{formulaEigenspace},
\begin{align*}
 \tau_k = \tau_{\lambda_k} : \ker (A - \lambda_k) \to \ran \Res_{\lambda_k}^\omega M, \quad u \mapsto \partial_\cL u |_\omega.
\end{align*}
Since by Proposition~\ref{formulaEigenspace} $\tau_k$ is an isomorphism, the formula
\begin{align*}
 \ker(A - \lambda_k) = \ran \left(\tau_k^{-1} \circ \Res_{\lambda_k}^\omega M \right)
\end{align*}
for the eigenspace corresponding to the eigenvalue $\lambda_k$ follows immediately. In particular, 
we can choose $\phi_1^{(k)}, \dots, \phi_{n(k)}^{(k)} \in H^{1/2}_\omega(\partial \Omega)$, $n(k)=\dim\ker(A-\lambda_k)$, such that the functions 
$$e_i^{(k)} = \tau_k^{-1}\bigl(\Res_{\lambda_k}^\omega M \bigr)\phi_{i}^{(k)}, \qquad 1 \leq i \leq n(k),$$ 
form an orthonormal basis in $\ker(A-\lambda_k)$. Then the orthogonal projection $E_k$ in $L^2(\Omega)$ onto $\ker (A - \lambda_k)$ is given by
\begin{align*}
 E_k u = \sum_{i = 1}^{n(k)} \big(u, e_i^{(k)} \big) e_i^{(k)}, \qquad u \in L^2(\Omega).
\end{align*}
Since the spectrum of $A$ consists only of eigenvalues with finite-dimensional eigen\-spaces 
we conclude that $A$ can be represented in the form
\begin{align*}
 A u = \sum_{k = 1}^\infty \lambda_k \sum_{i = 1}^{n(k)} \big(u, e_i^{(k)} \big) e_i^{(k)}, \qquad u \in \dom A.
\end{align*}
\end{proof}

\begin{remark}
Since $M$ is a holomorphic operator function whose values are bounded operators from 
$H^{1/2} (\partial \Omega) \to H^{-1/2} (\partial \Omega)$ it is sufficient to assume knowledge of $M(\lambda) \phi$ 
in Theorem~\ref{uniqueness} and Theorem~\ref{reconstruction} for $\phi$ in an arbitrary dense subspace of 
$H^{1/2}_\omega(\partial\Omega)$ and 
$\lambda$ in a discrete set of points $\cD$ with an accumulation point in $\rho(A)$. In particular, since 
the spectrum of $A$ is discrete, it is possible to choose $\cD$ as a sequence in any nonempty, open subset of $\R$.
\end{remark}

\section*{Appendix}

In this appendix we recall some definitions and basic facts on (unbounded) operators in Hilbert spaces and Banach space-valued 
mappings, which can be found in, e.g., the monographs \cite{AG93,C91,DS71,H82,RS80}.

\subsection*{Linear operators in Hilbert spaces}
Let $\cH$ and $\cK$ be Hilbert spaces with scalar products $(\cdot,\cdot)_\cH$ and $(\cdot,\cdot)_\cK$, and let $T$ be a 
linear operator from $\cH$ to $\cK$.
We write $\dom T$, $\ker T$ and $\ran T$ for the domain, kernel and range of $T$, respectively.
The operator $T$ is said to be {\it closed} if its graph is a closed subspace of $\cH\times\cK$. 
If  $\dom T$ is dense in $\cH$, then the
{\em adjoint operator} $T^*$ is defined by $T^*g:=\tilde g$, where
\begin{displaymath}
\dom T^*=\bigl\{g\in\cK:\text{exists}\,\,\tilde g\in\cH\,\,\text{such that}\,\,
(Tf,g)_\cK=(f,\tilde g)_\cH
\,\,\text{for all}\,\,f\in\dom T\bigr\}.
\end{displaymath}
Observe that $T^*$ is a closed linear operator from $\cK$ to $\cH$. Moreover, $\dom T^*$ is dense in $\cK$ if and only if 
the closure $\overline T$ of (the graph of) $T$ is an operator, and in this case $T^{**}$ coincides with  
$\overline T$. 

Let $A$ be
a linear operator in $\cH$. Then $A$ is said to be {\it symmetric} if the relation
$(Af,g)_\cH=(f,Ag)_\cH$
holds for all $f,g\in\dom A$. If $A$ is densely defined, then $A$ is symmetric if and only if 
the adjoint operator $A^*$ is an extension of $A$, that is, $\dom A\subseteq \dom A^*$
and $A^*f=Af$ holds for all $f\in\dom A$; in short $A\subseteq A^*$.
If $A=A^*$ holds, then the operator $A$ is called {\it selfadjoint}. Recall that a symmetric operator
$A$ is selfadjoint if and only if $\ran(A-\lambda_\pm)=\cH$ holds for some $\lambda_\pm\in\dC^\pm$ and that for
a closed symmetric operator $A$ and $\lambda\in\dC\setminus\dR$ the operator $(A-\lambda)^{-1}$ 
defined on the closed subspace $\ran(A-\lambda)$ is bounded by $\vert\Imag\,\lambda\vert^{-1}$.

Let $S$ be a closed linear operator in $\cH$. The {\it resolvent set} $\rho(S)$ consists of all $\lambda\in\dC$ such that
$(S-\lambda)^{-1}$ is a bounded operator defined on $\cH$. Note that the set $\rho(S)$ is open in $\dC$. The {\em spectrum} 
$\sigma(S)$ of $S$ is the complement of $\rho(S)$ in $\dC$, in particular, $\sigma(S)$ contains the eigenvalues of $S$, i.e., 
the points $\lambda\in\dC$ such that
$\ker(S-\lambda)\not=\{0\}$ holds. 
Recall that for a symmetric operator $A$ in $\cH$ the eigenvalues are real 
and that for a selfadjoint operator $A$ the whole spectrum $\sigma(A)$ is contained in $\dR$.
Moreover, by the spectral theorem the spectrum of a selfadjoint operator $A$ is empty if and only if the Hilbert space $\cH$ is trivial.

\subsection*{Holomorphic functions with values in Banach spaces}
Let $X$ be a Banach space and let $D\subseteq\dC$ be an open set.
If the function $m : D \to X$ is holomorphic on $D$ and $\mu \in \C$ is a pole of $m$, then the {\em residual} of $m$ at $\mu$ is 
\begin{align*}
 \Res_{\mu} m = \frac{1}{2 \pi i}  \int_{\Gamma_{\mu}} m(\zeta) \d \zeta,
\end{align*}
where $\Gamma_{\mu}$ denotes a closed Jordan curve with interior $\cO$ which contains $\mu$, such that 
$\cO \setminus \{\mu\}\subseteq D$. Equivalently, $\Res_{\mu} m$ is the first coefficient of negative 
order in the Laurent series expansion of $m$ at $\mu$. If $A$ is a selfadjoint operator in a Hilbert space 
$\cH$, then $\lambda \mapsto R_A(\lambda)=(A - \lambda)^{-1}$ is a holomorphic function on $\rho(A)$ 
with values in the space of bounded linear operators in $\cH$. Here each isolated point $\mu$ in $\sigma(A)$ 
is a pole of $R_A$ of order one and the orthogonal projection $E_\mu$ in $\cH$ onto the eigenspace $\ker(A - \mu)$ is
\begin{equation*}
E_\mu= - \Res_\mu R_A= - \frac{1}{2\pi i}\int_{\Gamma_\mu}(A-\zeta)^{-1} \d \zeta.
\end{equation*}

\end{document}